\documentclass[a4paper, svgnames, 12pt]{amsart}

\usepackage[a-1b]{pdfx}

\usepackage{
	amssymb,
	mathtools,
	xcolor,
	booktabs,
	multicol,
	tikz-cd,
}

\usepackage[english]{babel}

\usepackage[a4paper]{geometry}

\usepackage{tikz}
\usetikzlibrary{positioning, calc}

\usepackage{enumitem}
\setlist[enumerate]{label=\textup{(\arabic*)}}

\usepackage{hyperref}
\hypersetup{colorlinks, allcolors=MediumBlue}
\newcommand{\msc}[1]{\href{https://zbmath.org/classification/?q=#1}{#1}}
\newcommand{\doi}[1]{doi:\href{https://doi.org/#1}{#1}}

\usepackage[noabbrev, capitalise]{cleveref}
\crefdefaultlabelformat{#2\textup{#1}#3}

\crefname{main}	{Theorem}		{Theorems}
\crefname{thm}	{Theorem}		{Theorems}
\crefname{lem}	{Lemma}			{Lemmas}
\crefname{prop}	{Proposition}	{Propositions}
\crefname{dfn}	{Definition}	{Definitions}
\crefname{fig}	{Figure}		{Figures}
\crefname{tbl}	{Table}			{Tables}
\crefname{rmk}	{Remark}		{Remarks}
\crefname{exm}	{Example}		{Examples}

\theoremstyle{plain}
\newtheorem{main} {Theorem}

\newtheorem{thm}	{Theorem} [section]
\newtheorem{prop}	[thm] {Proposition}
\newtheorem{cor}	[thm] {Corollary}
\newtheorem{lem}	[thm] {Lemma}
\theoremstyle{definition}
\newtheorem{dfn}	[thm] {Definition}
\newtheorem{rmk}	[thm] {Remark}

\newtheorem*{thx}{Acknowledgements}

\numberwithin{equation}{section}

\newcommand{\F}{\mathbb{F}}
\newcommand{\Dih}{\mathrm{D}}
\newcommand{\Quat}{\mathrm{Q}}
\newcommand{\Semi}{\mathrm{S}}
\newcommand{\aug}{\varepsilon_N}
\newcommand{\II}{\Delta}
\newcommand{\FF}{F}

\DeclarePairedDelimiterX\set[1]\lbrace\rbrace{\,\def\given{\mid}#1\,} % usage: \set{ 2n \given n > 0 }, \set[\big]{ \frac{p}{q} \given p, q > 0 }
\DeclarePairedDelimiterX\gen[1]\langle\rangle{\,\def\given{\mid}#1\,} % usage: \gen{ g \given g^n = 1 }
\newcommand{\Ideal}{\Delta}
\newcommand{\Comm}{\gamma}
\newcommand{\Agemo}{\mho}
\newcommand{\Zeta}{\mathrm{Z}}
\DeclareMathOperator{\Ext}{Ext}
\DeclareMathOperator{\Soc}{\Psi}
\DeclareMathOperator{\Frat}{\Phi}
\renewcommand{\dim}[1]{\operatorname{dim} #1}
\newcommand{\codim}[1]{\operatorname{codim} #1}
\newcommand{\cohomology}{\operatorname{H}^*}
\DeclareMathOperator{\SG}{\textup{\texttt{SG}}}
\DeclareMathOperator{\cc}{cc}
\DeclareMathOperator{\dg}{d}

\makeatletter
\@namedef{subjclassname@2020}{%
  \textup{2020} Mathematics Subject Classification}
\makeatother

\newenvironment{example}[1][]{\refstepcounter{thm}\par\medskip
   \noindent \textbf{Example~\thethm. #1} \rmfamily}{\medskip}

\begin{document}

\title[Invariants and a reduction for the MIP]{Abelian invariants and a reduction theorem for the modular isomorphism problem}

% LM
\author[L.~Margolis]{Leo Margolis}
\address[Leo Margolis]{Institudo de Ciencias Matem\'aticas, C/ Nicolas Cabrera 13, 28049 Madrid, Spain.}
\email{leo.margolis@icmat.es}

% TS
\author[T.~Sakurai]{Taro Sakurai}
\address[Taro Sakurai]{Department of Mathematics and Informatics, Graduate School of Science, Chiba University, 1-33, Yayoi-cho, Inage-ku, Chiba-shi, Chiba, 263-8522, Japan.}
\email{tsakurai@math.s.chiba-u.ac.jp}

% MS
\author[M.~Stanojkovski]{Mima Stanojkovski}
\address[Mima Stanojkovski]{ 
Universit\`a di Trento, Dipartimento di Matematica, via Sommarive 14, 38123 Trento (TN), Italy.
}
\email{mima.stanojkovski@unitn.it}

\thanks{
Leo Margolis acknowledges financial support from the Spanish Ministry of Science and Innovation, through the “Severo Ochoa Programme for Centres of Excellence in R\&D” (CEX2019-000904-S). 
Mima Stanojkovski was supported by the Deutsche Forschungsgemeinschaft (DFG, German Research Foundation) – Project-ID 286237555 – TRR 195.}

\subjclass[2020]{\msc{16S34} (\msc{20C05}, \msc{20D15})}

\keywords{Abelian invariant, modular group algebra, modular isomorphism problem, power structure, reduction}

\date{\today}

\begin{abstract}
 We show that elementary abelian direct factors can be disregarded in the study of the modular isomorphism problem. Moreover, we obtain four new series of abelian invariants of the group base in the modular group algebra of a finite $p$-group. Finally, we apply our results to new classes of groups.
\end{abstract}

\maketitle

\section*{Introduction}

Given a field $\FF$ of positive characteristic $p$ and a finite $p$-group $G$,
the modular group algebra $\FF G$ of $G$ over $\FF$ plays a  fundamental role in studying linear representations of $G$ over $\FF$.
The interplay between the group structure of $G$ and the algebra structure of $\FF G$ has been thoroughly studied and is yet still not completely understood.
It is obvious that isomorphic finite $p$-groups $G \cong H$ define isomorphic finite-dimensional algebras $\FF G \cong \FF H$.
The \emph{modular isomorphism problem} (MIP) asks whether the converse holds. In symbols, this reads
\[
\FF G\cong \FF H \ \Longrightarrow \ G\cong H?
\]
Over the decades, the last question has been positively answered
for  finite $p$-groups of small orders and for special classes of them (we refer to \cite{San85,HS06, EK11} for an overview of most known results, while more recent contributions are \cite{BK19, Sak20, MM20, BdR20, MS21}).
However, a quite recent breakthrough \cite{GLMdR22} found the first counterexample to the modular isomorphism problem for the case $p = 2$.
Having said that, the modular isomorphism problem is still open for odd primes and it is far from being solved in general. For example, as of today, no structural reduction theorems are known.

In this article we prove the following result reducing the modular isomorphism problem to finite $p$-groups free of elementary abelian direct factors.
\begin{main}\label{main:reduction}
Let $\FF$ be a field of positive characteristic $p$ and  let $G$ and $H$ be finite $p$-groups. Let, moreover, $E$ be a finite elementary abelian $p$-group.
Then an algebra isomorphism $\FF[E \times G]\cong \FF[E \times H]$ implies an algebra isomorphism $\FF G \cong \FF H$,
in symbols
\begin{equation*}
	\FF[E \times G] \cong \FF[E \times H] \implies \FF G \cong \FF H.
\end{equation*}
\end{main}
\cref{main:reduction} is a simplified instance of \cref{thm:reduction} and is proven in \cref{sec:reduction}.
One ingredient in the proof of the reduction theorem is that the isomorphism type of the intersection
of the socle  with the  Frattini subgroup of $G$ is an invariant of $\FF G$.  This last invariant can be derived from the following \cref{main:invariants}; see in particular \cref{cor:invariants}.

Let $\Zeta(G)$ and $\Comm(G)$ denote the center and commutator subgroup of $G$, respectively.
Moreover, for a non-negative integer $n$, set
\[
\Omega_n(G)=\gen{ g\in G \given g^{p^n}=1 } \textup{ and }
\Agemo_n(G)=\gen{ g^{p^n} \given g\in G }.
\]

\begin{main}\label{main:invariants}
Let $\FF$ be a field of positive characteristic $p$ and  let $G$ be a finite $p$-group. Let, moreover, $n$ be a non-negative integer.
	Then the isomorphism types of the following are invariants of $\FF G$:
	\begin{multicols}{2}
		\begin{enumerate}
			\item\label{it:invariants1} $G/\Comm(G)\Omega_n(\Zeta(G))$,
			\item\label{it:invariants2} $\Comm(G)\Omega_n(\Zeta(G))/\Comm(G)$,
			\columnbreak
			\item\label{it:invariants3} $\Zeta(G)\cap\Agemo_n(G)\Comm(G)$,
			\item\label{it:invariants4} $\Zeta(G)/\Zeta(G)\cap\Agemo_n(G)\Comm(G)$.
		\end{enumerate}
	\end{multicols}
\end{main}

\begin{figure}[htbp]
	\includegraphics{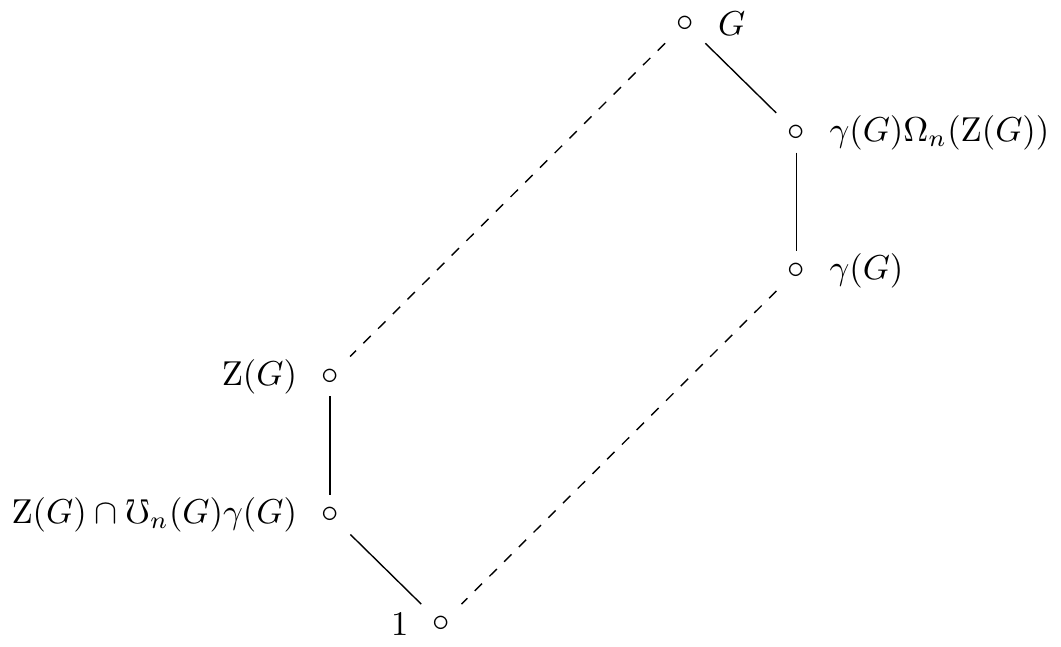}
	\caption{The solid lines depict the new abelian invariants of the modular group algebra in \cref{main:invariants}.}
	\label{fig:duality}
\end{figure}

The search for invariants is, if possible, even more motivated by the discovery of a counterexample to the modular isomorphism problem. Indeed, such invariants provide obstructions to the coexistence of two non-isomorphic group bases in the same modular group algebra.

We apply our results to obtain positive results on the modular isomorphism problem for two new classes of finite $2$-groups both of which share certain properties with the known counterexamples; cf.\ \cref{rmk:counterex}. 

We stress that all our results are independent of the choice of the field, in particular, they are not confined to the case where $\FF=\F_p$ is the field of $p$ elements (note that this is the case for several existing results on the modular isomorphism problem). Moreover, our results would be new also in that framework.

This article is organized as follows.
In \cref{sec:prel} we introduce the notation and recall some standard results.
\cref{sec:invariants} is devoted to establishing four new series of abelian invariants of group bases of modular group algebras, the main results being Theorems~\ref{thm:CO*} and \ref{thm:ZA*}.
\cref{sec:decomp} deals with elementary decompositions of finite $p$-groups from a purely group-theoretic point of view.
The reduction, Theorem~\ref{thm:reduction}, is proved in \cref{sec:reduction} using one of the new abelian invariants and the decomposition from \cref{sec:decomp}.
Two concrete applications to the modular isomorphism problem, which are independent of each other, are given in \cref{sec:appl}.

\begin{thx}
We thank \'Angel del R\'{\i}o for his helpful comments on an early draft of this manuscript. We also thank Max Horn and Eamonn O'Brien for some clarifications regarding the classification of groups of order $32$.
We thank the anonymous referee for their careful reading of this manuscript and the many useful comments. We thank Diego Garc\'{\i}a-Lucas for spotting an imprecision in the original formulation of \cref{lem:filt} and for continuing this story in his recent preprint \cite{Gar22}. 
\end{thx}

\section{Preliminaries and notation}
\label{sec:prel}

Throughout the whole paper, $p$ denotes a prime number, $G$ and $H$ finite $p$-groups and $\FF$ a field of characteristic $p$. 

\subsection{Groups}

Denote by $\Frat(G)$ the Frattini subgroup of $G$ and by $\Soc(G)$ the socle of $G$, i.e.\ the subgroup of $G$ that is generated by central elements of prime order.
The set of conjugacy classes of $G$ is denoted $\cc(G)$ and for $g,h \in G$ we write $[g,h] = g^{-1}h^{-1}gh$ for the commutator of $g$ and $h$. The commutator subgroup of $G$ is denoted by $\Comm(G)$, while $\Zeta(G) $ denotes the center of $G$. Moreover, for every non-negative integer $n$, we define the following crucial players of this paper:
\begin{multicols}{2} \begin{enumerate}
\item   $\Agemo_n(G)=\gen{ g^{p^n} \given g\in G }$,
	\item $\Agemo^*_n(G) = \Agemo_n(G)\Comm(G)$,
	\item $\Omega_n(G)=\gen{ g\in G \given g^{p^n}=1 }$,
	\item $\Omega^*_n(G) = \Omega_n(\Zeta(G))$.
\end{enumerate} \end{multicols}
\noindent
For instance, with this notation, we have that  
$$\Agemo_n(G/\Comm(G)) = \Agemo^*_n(G)/\Comm(G), \ \Frat(G) = \Agemo^*_1(G), \textup{ and } \Soc(G) = \Omega^*_1(G).$$
The following result is a direct consequence of the classification of finite abelian $p$-groups.

\begin{prop}\label{prop:type}
Assume that $G$ and $H$ are abelian. Then $G$ and $H$ are isomorphic if and only if $(|\Agemo_n(G)|)_{n \geq 0}=(|\Agemo_n(H)|)_{n \geq 0}$.
\end{prop}
\begin{proof}
	This readily follows from \cite[Chapter II, (1.4)]{Mac95}.
\end{proof}

\subsection{Algebras}
Let $A$ be a finite-dimensional algebra over $\FF$.
Analogously to the notation for groups, we let $\Comm(A) = \sum_{x, y \in A} \FF(xy -yx)$ denote the commutator subspace of $A$ and $\Zeta(A)$ the center of $A$. 
We use the following notation for a non-negative integer $n$ and a subspace $X$ of $A$:
\begin{multicols}{2} \begin{enumerate}
	\item $\Agemo_n(X) = \set{ x^{p^n} \given x \in X }$,
	\item $\Omega_n(X) = \set{ x \in X \given x^{p^n} = 0 }$.
\end{enumerate} \end{multicols}
\noindent
The codimension of $X$ is $\codim{X} = \dim{A} - \dim{X}$, where the dimensions are taken as vector spaces over $\FF$.

The following fact on power maps can be found in \cite[Chapter 2, Lemma 3.1]{Pas77}.
\begin{lem}
	\label{lem:dream}
	Let $A$ be a finite-dimensional algebra over $\FF$ and $n$ a non-negative integer.
	Then
	\begin{equation*}
		(x + y)^{p^n} \equiv x^{p^n} + y^{p^n}  \bmod \Comm(A)
	\end{equation*}
	 for every pair of elements $x$ and $y$ in $A$.
\end{lem}

\subsection{Group algebras}

The group $H$ is a \emph{group base} of $\FF G$ if $H$ consists of units in $\FF G$ and $H$ is a basis of $\FF G$.
An object $\mathcal{P}(G)$ associated with $G$ is said to be an \emph{invariant} of $\FF G$
if $\FF G \cong \FF H$ implies that $\mathcal{P}(G)=\mathcal{P}(H)$.
A subset $X(\FF G)$ of $\FF G$ is said to be \emph{canonical} in $\FF G$
if there is a formula $\varphi(x)$ in the language of algebras such that $X(\FF G)$ consists of the elements of $\FF G$ satisfying $\varphi(x)$.
Then every algebra isomorphism $\FF G \to \FF H$ maps $X(\FF G)$ to $X(\FF H)$.
The {\em augmentation ideal} of $\FF G$ is
\begin{equation*}
	\Ideal(\FF G) = \bigoplus_{\substack{g \in G\\ g \neq 1}} \FF(g - 1).
\end{equation*}
It is well-known that the augmentation ideal equals the unique maximal ideal in $\FF G$, cf.\ \cite[Chapter 8, Lemma 1.17]{Pas77}.
In particular, $\Ideal(\FF G)$ is canonical in the modular group algebra. Finally, for a conjugacy class $\kappa$ of $G$, we define the \emph{class sum} of $\kappa$ in $\FF G$ as
\[
	\hat\kappa = \sum_{g \in \kappa} g.
\]
The following facts on augmentation ideals relative to normal subgroups can be found in \cite[Chapter 1, Lemma 1.8]{Pas77}.

\begin{lem}
	\label{lem:quotient}
	Let $N$ be a normal subgroup of $G$.
	Then $\Ideal(\FF N)\FF G$ is the kernel of the natural surjective map $\aug \colon \FF G \rightarrow \FF [G/N]$.
\end{lem}

The augmentation ideal relative to the commutator subgroup of $G$ plays a special role in $\FF G$, as emphasized by the following lemma.

\begin{lem}
	\label{lem:commutative}
	Let $I$ be an ideal of $\FF G$.
	Then the  following are equivalent:
	\begin{enumerate}
		\item\label{item:comm1} $\FF G/I$ is commutative.
		\item\label{item:comm2} $I \supseteq \Comm(\FF G)$.
		\item\label{item:comm3} $I \supseteq \Ideal(\FF \Comm(G))$.
	\end{enumerate}
\end{lem}
\begin{proof}
To see that \ref{item:comm1} implies \ref{item:comm2} note that $\FF G/I$ being commutative implies that $I$ contains all elements of shape $xy -yx$ for elements $x$ and $y$ of $\FF G$. These elements are exactly the generators of $\Comm(\FF G)$. Moreover, \ref{item:comm2} and \ref{item:comm3} are equivalent by \cite[Lemma 3.5]{San85}. Finally, \ref{item:comm3} implies \ref{item:comm1} by Lemma~\ref{lem:quotient}.
\end{proof}

\begin{lem}
	\label{lem:product}
	Let $I$ be an ideal of $\FF G$.
	Let, moreover, $K$ and $L$ be normal subgroups of $G$.
	Then the following are equivalent:
	\begin{enumerate}
		\item\label{item:KxL} $I \supseteq \Ideal(\FF KL)$.
		\item\label{item:K+L} $I \supseteq \Ideal(\FF K)$ and $I \supseteq \Ideal(\FF L)$.
	\end{enumerate}
\end{lem}
\begin{proof}
The implication $\ref{item:KxL} \Rightarrow \ref{item:K+L}$ is clear.
	To see that \ref{item:K+L} implies \ref{item:KxL}, note that, for any choice of $x, y \in G$, one has $xy - 1 = (x - 1)(y - 1) + (x - 1) + (y - 1)$.
\end{proof}

\begin{lem}
	\label{lem:center} The following facts hold in $\FF G$:
	\begin{enumerate}
	\item \label{item:KJ2} $\Comm(\FF G) \subseteq \Ideal(\FF G)^2$,\vspace{5pt}
	\item \label{item:Z} $\Zeta(\FF G)
			= \bigoplus_{\kappa \in \cc(G)} \FF \hat\kappa
			= \bigoplus_{z \in \Zeta(G)} \FF z \oplus \bigoplus_{\substack{\kappa \in \cc(G)\\ |\kappa| \neq 1}} \FF \hat\kappa$, \vspace{5pt}
	\item \label{item:ZK} $\Zeta(\FF G) \cap \Comm(\FF G)
			= \bigoplus_{\substack{\kappa \in \cc(G)\\ |\kappa| \neq 1}} \FF \hat\kappa$,
	\item \label{item:ideal} $\Zeta(\FF G) \cap \Comm(\FF G)$ is an ideal of $\Zeta(\FF G)$,
	\item \label{item:FZG} $\Zeta(\FF G) = \FF\Zeta(G) \oplus (\Zeta(\FF G) \cap \Comm(\FF G))$.
	\end{enumerate}
\end{lem}
\begin{proof}

\ref{item:KJ2} Note that, for every $x, y \in \FF G$,  one has
$$xy - yx = (x - 1)(y - 1) - (y - 1)(x - 1).$$
The claim now follows from $\FF G = \FF \oplus \Ideal(\FF G)$.
Points \ref{item:Z} to \ref{item:FZG} are direct consequences of  \cite[Section~III.6]{Seh78}.
\end{proof}

Finally, we state some classical results on the modular isomorphism problem: the invariance of the center due to Ward and Sehgal \cite[Chapter~III, Theorem~6.6]{Seh78} and the invariance of the Frattini quotient \cite[Chapter~14, Lemma~2.7]{Pas77}.

\begin{thm}
	\label{thm:WS}
	The isomorphism type of the center $\Zeta(G)$ is an invariant of $\FF G$.
\end{thm}

\begin{prop}
	\label{prop:Frat}
	The isomorphism type of the Frattini quotient $G/\Frat(G)$ is an invariant of $\FF G$.
\end{prop}

\section{Abelian invariants}
\label{sec:invariants}

In this section we present new classes of abelian invariants of the modular group algebra of a finite $p$-group, namely those from \cref{main:invariants}.
This theorem will be obtained from the combination of \cref{thm:CO*,thm:ZA*}, which
are dual to each other in the following sense. While \cref{thm:CO*} is a result of filtering the abelianization $G/\Comm(G)$ taking products with the terms of $(\Omega_n^*(G))_{n\geq 0}$, in \cref{thm:ZA*}  the center $\Zeta(G)$ is filtered via intersections with the terms of $(\Agemo_n^*(G))_{n\geq 0}$.

\subsection{Commutator subgroups and socles}\label{sec:CO*}
We now establish the first pair of abelian invariants from \cref{main:invariants}.
We will do so relying on \cref{prop:type} and identifying canonical ideals in the modular group algebra reflecting properties of the considered group quotients: the analogy is drawn in \cref{fig:CO*}.

\begin{thm}
	\label{thm:CO*}
	Let $n$ be a non-negative integer. Then the isomorphism types of the following are invariants of $\FF G$:
	\begin{multicols}{2} \begin{enumerate}
		\item\label{it:CO*1} $G/\Comm(G)\Omega^*_n(G)$,
		\item\label{it:CO*2} $\Comm(G)\Omega^*_n(G)/\Comm(G)$.
	\end{enumerate} \end{multicols}
\end{thm}

\begin{figure}[h]
	\includegraphics{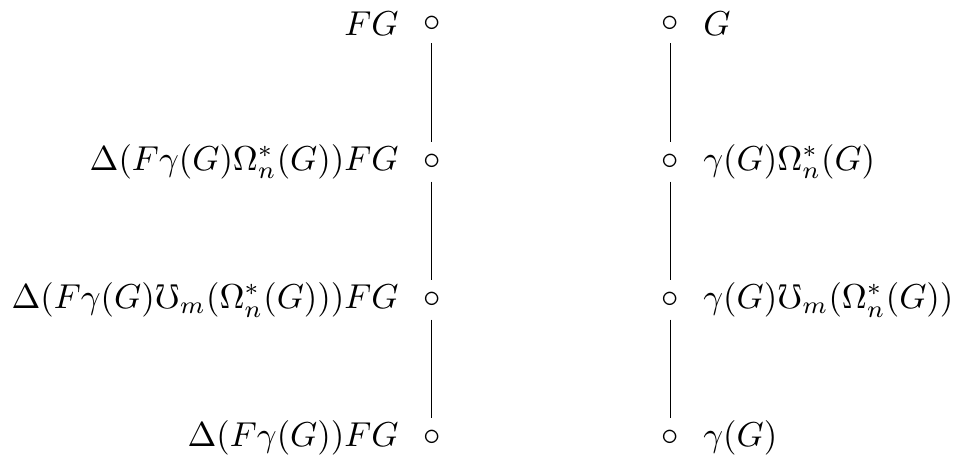}
	\caption{Corresponding ideals and subgroups.}\label{fig:CO*}
\end{figure}

\begin{lem}
	\label{lem:powers}
	Let $m$ and $n$ be non-negative integers.
	Assume that $G$ is abelian.
	Then the following equality holds:
	\begin{equation*}
		\Ideal(\FF \Agemo_m(\Omega_n(G)))\FF G
		= \Agemo_m(\Omega_n(FG))\FF G.
	\end{equation*}
\end{lem}
\begin{proof}
Thanks to \cref{lem:dream},	it is easy to show that \[\Ideal(\FF \Agemo_m(\Omega_n(G)))\FF G
		\subseteq \Agemo_m(\Omega_n(FG))\FF G.\]
	To show the opposite inclusion, take an element $x \in \Omega_n(\FF G)$. We will show that $x^{p^m} \in \Ideal(\FF \Agemo_m(\Omega_n(G)))\FF G$.
To this end, choose a complete set $R$ of coset representatives of $\Omega_n(G)$ in $G$	and write
	\[
		x = \sum_{g \in \Omega_n(G)}\sum_{r \in R} \alpha_{gr} gr\ \textup{ for some }\ \alpha_{gr} \in \FF.
	\]
	Observe now that the map $R \rightarrow G$, defined by $r \mapsto r^{p^n}$, is injective.
	From
	\begin{align*}
		0 = x^{p^n}
		= \bigg( \sum_{g \in \Omega_n(G)}\sum_{r \in R} \alpha_{gr} gr \bigg)^{p^n}
		= \sum_{r \in R} \bigg(\sum_{g \in \Omega_n(G)} \alpha_{gr} \bigg)^{p^n} r^{p^n},
	\end{align*}
	we obtain, for each representative $r \in R$, that
	$
		\sum_{g \in \Omega_n(G)} \alpha_{gr} = 0.
	$
	Hence
	\begin{align*}
		x
		= \sum_{r \in R}\bigg(\sum_{g \in \Omega_n(G)} \alpha_{gr} gr  - \sum_{g \in \Omega_n(G)} \alpha_{gr} r \bigg)
		= \sum_{r \in R}\sum_{g \in \Omega_n(G)} \alpha_{gr} (g - 1)r
	\end{align*}
	and we see that
	$
		x^{p^m} = \sum_{r \in R}\sum_{g \in \Omega_n(G)} \alpha_{gr}^{p^m} (g^{p^m} - 1)r^{p^m} \in \Ideal(\FF \Agemo_m(\Omega_n(G)))\FF G
	$.
\end{proof}

\begin{lem}
	\label{lem:augCAO*}
	Let $m$ and $n$ be non-negative integers.
	Then
	\[ \Ideal(\FF \Comm(G)\Agemo_m(\Omega^*_n(G)))\FF G \]
	is the smallest ideal of $\FF G$ containing both $\Comm(\FF G)$ and $\Agemo_m(\Omega_n(\Zeta(\FF G)))$.
\end{lem}
\begin{proof}
	Let $I$ be the smallest ideal of $\FF G$ containing $\Comm(\FF G)$ and $\Agemo_m(\Omega_n(\Zeta(\FF G)))$ and set $K = \Comm(G)\Agemo_m(\Omega^*_n(G))$: we show that $I=\Ideal(\FF K)\FF G$.

	With the aid of \cref{lem:commutative,lem:product}, it is easy to show that $I \supseteq \Ideal(\FF K)\FF G$.
	To prove that $I \subseteq \Ideal(\FF K)\FF G$,
	we will use \cref{lem:commutative,lem:product} to show that $\Agemo_m(\Omega_n(\Zeta(\FF G))) \subseteq \Ideal(\FF K)\FF  G$.
	To this end, fix $x \in \Omega_n(\Zeta(\FF G))$: we will prove that $x^{p^m} \in \Ideal(\FF K)\FF G$.
	As a consequence of \cref{lem:center}\ref{item:Z}, write
	\[
		x = \sum_{z \in \Zeta(G)} \alpha_z z + \sum_{\substack{\kappa \in \cc(G)\\ |\kappa| \neq 1}} \beta_\kappa \hat\kappa
	\]
	for some $\alpha_z$, $\beta_\kappa \in \FF$.
	\cref{lem:dream} ensures that 
	\[
		0 = x^{p^n} = \bigg( \sum_{z \in \Zeta(G)} \alpha_z z \bigg)^{p^n} + \bigg( \sum_{\substack{\kappa \in \cc(G)\\ |\kappa| \neq 1}} \beta_\kappa \hat\kappa \bigg)^{p^n}
	\]
	and so, thanks to \cref{lem:center}\ref{item:Z}--\ref{item:ideal}, we obtain that $\sum_{z \in \Zeta(G)} \alpha_z z \in \Omega_n(\FF \Zeta(G))$.
	Observe now that
	\[
		x^{p^m} = \bigg( \sum_{z \in \Zeta(G)} \alpha_z z \bigg)^{p^m} + \bigg( \sum_{\substack{\kappa \in \cc(G)\\ |\kappa| \neq 1}} \beta_\kappa \hat\kappa \bigg)^{p^m}.
	\]
	The first summand belongs to $\Ideal(\FF \Agemo_m(\Omega^*_n(G)))\FF G$, by \cref{lem:powers}, and the second summand belongs to $\Ideal(\FF \Comm(G))\FF G$, by \cref{lem:center}\ref{item:ZK}--\ref{item:ideal}.
Now	\cref{lem:product} yields that  $x^{p^m} \in \Ideal(\FF K)\FF G$.
\end{proof}

We close the present section with the proof of \cref{thm:CO*}.
\begin{proof}[Proof of \cref{thm:CO*}]
Assume that $\FF G\cong \FF H$ and let $n$ be a non-negative integer.

\ref{it:CO*1} Thanks to \cref{lem:augCAO*}, the ideal $\Ideal(\FF \Comm(G)\Omega^*_n(G))\FF G$ is canonical in $\FF G$ and thus every algebra isomorphism $\FF G\rightarrow \FF H$ induces an algebra isomorphism
$$\FF G/\Ideal(\FF \Comm(G)\Omega^*_n(G))\FF G\rightarrow \FF H/\Ideal(\FF \Comm(H)\Omega^*_n(H))\FF H.$$
Thanks to \cref{lem:quotient}, the quotient $\FF G/\Ideal(\FF \Comm(G)\Omega^*_n(G))\FF G$ is isomorphic to $\FF[G/\Comm(G)\Omega^*_n(G)]$ and so it follows from \cref{thm:WS} that the isomorphism type of  $G/\Comm(G)\Omega^*_n(G)$ is an invariant of $\FF G$.

\ref{it:CO*2}	Combining \cref{lem:augCAO*,lem:quotient} with the fact that
	\begin{equation*}
		\Agemo_m(\Comm(G)\Omega^*_n(G)/\Comm(G)) = \Comm(G)\Agemo_m(\Omega^*_n(G))/\Comm(G),
	\end{equation*}
	the sequence $(|\Agemo_k(\Comm(G)\Omega^*_n(G)/\Comm(G))|)_{k \geq 0}$ is an invariant of $\FF G$.
	\cref{prop:type} yields that the isomorphism type of $\Comm(G)\Omega^*_n(G)/\Comm(G)$ is an invariant of $\FF G$. This concludes the proof.
\end{proof}

\subsection{Frattini subgroups and centers}\label{sec:ZA*}
We now establish the second pair of abelian invariants from \cref{main:invariants}. In analogy to what we did in \cref{sec:CO*}, we refer the reader to \cref{fig:ZA*} for an overview of the ideals involved in the proof\footnote{Technically speaking, we will not consider the algebra $\FF \Zeta(G)$, but rather an isomorphic algebra $\Zeta(\FF G)/(\Zeta(\FF G) \cap \Comm(\FF G))$.}.

\begin{thm}
	\label{thm:ZA*}
	Let $n$ be a non-negative integer. Then the isomorphism types of the following are invariants of $\FF G$:
	\begin{multicols}{2} \begin{enumerate}
		\item\label{it:ZA*1}  $\Zeta(G) \cap \Agemo^*_n(G)$,
		\item\label{it:ZA*2}  $\Zeta(G)/\Zeta(G) \cap \Agemo^*_n(G)$.
	\end{enumerate} \end{multicols}
\end{thm}

\begin{figure}[h]
	\includegraphics{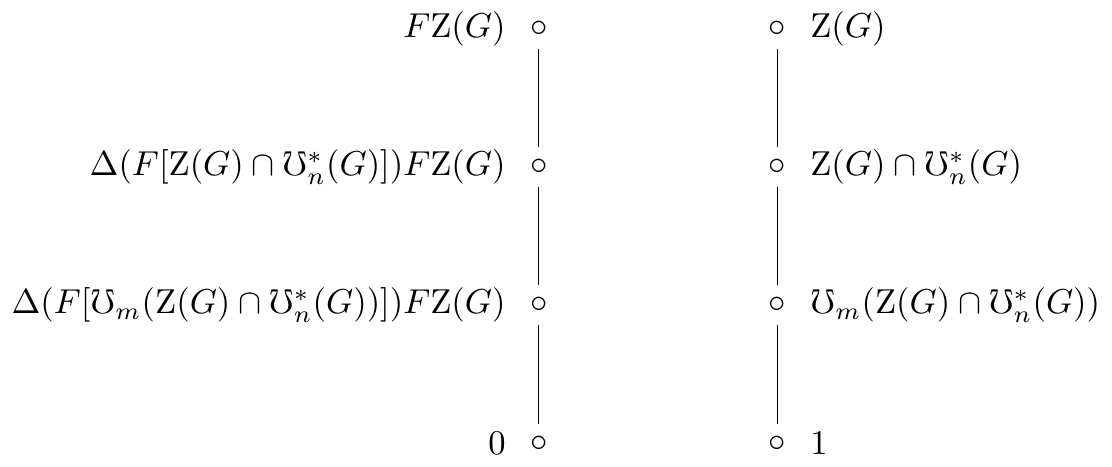}
	\caption{Corresponding ideals and subgroups.}\label{fig:ZA*}
\end{figure}

\begin{lem}
	\label{lem:augA}
	Assume that $G$ is abelian. 
	Let $K$ be a subgroup of $G$ and $m$ a non-negative integer.  
	Then
	\[ \Ideal(\FF \Agemo_m(K))\FF G \]
	is the smallest ideal of $\FF G$ containing $\Agemo_m(\Ideal(\FF K)\FF G).$
\end{lem}
\begin{proof}
	This is a direct consequence of \cref{lem:dream}.
\end{proof}

\begin{lem}
	\label{lem:augA*}
	Let $n$ be a non-negative integer.
	Then
	\[ \Ideal(\FF \Agemo^*_n(G))\FF G \]
	is the smallest ideal of $\FF G$ containing both $\Comm(\FF G)$ and $\Agemo_n(\Ideal(\FF G))$.
\end{lem}
\begin{proof}
	This follows easily from \cref{lem:product,lem:commutative,lem:dream}.
\end{proof}

\begin{lem}
	\label{lem:ZnN}
	Let $N$ be a  subgroup of $G$ containing $\Comm(G)$.
	Then the following equality holds:
	\begin{equation}\label{eq:SSum}
		\Zeta(\FF G) \cap \Ideal(\FF N)\FF G
		= \Ideal(\FF[\Zeta(G) \cap N])\FF \Zeta(G) \oplus (\Zeta(\FF G) \cap \Comm(\FF G)).
	\end{equation}
\end{lem}
\begin{proof}
	Note that $N$ is normal in $G$ because it contains $\Comm(G)$.
	It is easy to show that the left-hand side contains the right-hand side.
	To show the opposite inclusion, fix $x\in\Zeta(\FF G) \cap \Ideal(\FF N)\FF G$ and, thanks to \cref{lem:center}\ref{item:Z}, write
	\[
		x = \sum_{z\in\Zeta(G)} \alpha_z z + \sum_{\substack{\kappa\in\cc(G)\\ |\kappa|\neq1}} \beta_{\kappa} \hat{\kappa}
	\]
	for some $\alpha_z$, $\beta_\kappa \in \FF$.
	By \cref{lem:center}\ref{item:ZK}, the second summand belongs to $\Zeta(\FF G) \cap \Comm(\FF G)$; we will show that the first summand belongs to $\Ideal(\FF[\Zeta(G) \cap N])\FF \Zeta(G)$.
	Recall that $\Ideal(\FF N)\FF G$ is the kernel of the natural surjective map $\aug \colon \FF G \to \FF [G/N]$ by \cref{lem:quotient}.
	In particular $\aug(x) = 0$ yields
	\begin{equation*}
		\sum_{z\in\Zeta(G)} \alpha_z \aug(z) + \sum_{\substack{\kappa\in\cc(G)\\ |\kappa|\neq1}} \beta_{\kappa} \aug(\hat{\kappa})=0.
	\end{equation*}
	Since $N$ contains $\Comm(G)$, each non-central conjugacy class $\kappa$ satisfies
	$
		\aug(\hat{\kappa}) = 0.
	$
	It follows that
	$
		\sum_{z\in\Zeta(G)} \alpha_z zN = 0
	$.
	Choose a complete set $R$ of coset representatives of $\Zeta(G) \cap N$ in $\Zeta(G)$.
	As the map $R \to G/N$, defined by $r \mapsto rN$ is injective, from
	\begin{align*}
		0
		&= \sum_{z\in\Zeta(G)} \alpha_z zN
		= \sum_{r \in R}\sum_{n\in\Zeta(G)\cap N} \alpha_{rn} {rn}N
		= \sum_{r \in R} \bigg( \sum_{n\in\Zeta(G)\cap N} \alpha_{rn} \bigg) rN,
	\end{align*}
	we deduce, for each representative $r \in R$, that $\sum_{n\in\Zeta(G)\cap N} \alpha_{rn} = 0$.
	Hence
	\begin{align*}
		\sum_{z\in\Zeta(G)} \alpha_z z
		&= \sum_{r \in R}\sum_{n\in\Zeta(G)\cap N} \alpha_{rn} {rn} \\
		&= \sum_{r \in R}\bigg( \sum_{n\in\Zeta(G)\cap N} \alpha_{rn} {rn} - \sum_{n\in\Zeta(G)\cap N} \alpha_{rn} r \bigg) \\
		&= \sum_{r \in R}\sum_{n \in \Zeta(G) \cap N} \alpha_{nr} (n - 1)r
	\end{align*}
	and we see that $\sum_z \alpha_z z \in \Ideal(\FF[\Zeta(G) \cap N])\FF \Zeta(G)$.

	Note that the sum in \eqref{eq:SSum} is direct as consequence of \cref{lem:center}\ref{item:Z}.
\end{proof}

We close \cref{sec:ZA*} with the proof of \cref{thm:ZA*}.
\begin{proof}[Proof of \cref{thm:ZA*}]
Assume that $\FF G \cong \FF H$ and let $n$ and $m$ be non-negative integers.
We start by showing that the isomorphism type of the quotient group $\Zeta(G)/\Agemo_m(\Zeta(G) \cap \Agemo^*_n(G))$ is an invariant of $\FF G$.

Recall that $\Zeta(\FF G) \cap \Comm(\FF G)$ is an ideal of $\Zeta(\FF G)$ by \cref{lem:center}\ref{item:ideal}.
Let $\Theta(\FF G)$ be the smallest ideal of $\Zeta(\FF G)$ containing both
$\Agemo_m(\Zeta(\FF G) \cap \Ideal(\FF \Agemo^*_n(G))\FF G)$
and
$\Zeta(\FF G) \cap \Comm(\FF G)$
so that
\begin{equation}\label{eq:Theta}
    \Theta(\FF G) = \Agemo_m(\Zeta(\FF G) \cap \Ideal(\FF \Agemo^*_n(G))\FF G)\Zeta(\FF G) + (\Zeta(\FF G) \cap \Comm(\FF G)).
\end{equation}
Since, by \cref{lem:augA*}, the ideal $\Ideal(\FF \Agemo^*_n(G))\FF G$ is canonical in $\FF G$ so is the ideal $\Theta(\FF G)$.
Thus every algebra isomorphism $\FF G \to \FF H$ induces an algebra isomorphism $\Zeta(\FF G)/\Theta(\FF G) \to \Zeta(\FF H)/\Theta(\FF H)$.

Thanks to \cref{lem:ZnN}, the following equality holds:
\begin{equation}\label{eq:subtle}
	\Zeta(\FF G) \cap \Ideal(\FF \Agemo^*_n(G))\FF G = \Ideal(\FF[\Zeta(G) \cap \Agemo^*_n(G)])\FF \Zeta(G) \oplus (\Zeta(\FF G) \cap \Comm(\FF G)).
\end{equation}
We rewrite $\Theta(\FF G)$ as follows.
Substituting \eqref{eq:subtle} into \eqref{eq:Theta} yields
\begin{align*}
    \Theta(\FF G)
    &= \Agemo_m(\Ideal(\FF[\Zeta(G) \cap \Agemo^*_n(G)])\FF \Zeta(G) \oplus (\Zeta(\FF G) \cap \Comm(\FF G)))\Zeta(\FF G) \\
    & \qquad + (\Zeta(\FF G) \cap \Comm(\FF G))
\end{align*}
and expanding the $p^m$-th powers on the right hand side yields
\begin{equation*}
    \Theta(\FF G) = \Agemo_m(\Ideal(\FF[\Zeta(G) \cap \Agemo^*_n(G)])\FF \Zeta(G))\Zeta(\FF G) + (\Zeta(\FF G) \cap \Comm(\FF G)).
\end{equation*}
Using $\Zeta(\FF G) = \FF \Zeta(G) \oplus (\Zeta(\FF G) \cap \Comm(\FF G))$ given in \cref{lem:center}\ref{item:FZG}, we get
\begin{equation*}
\Theta(\FF G) = \Agemo_m(\Ideal(\FF[\Zeta(G) \cap \Agemo^*_n(G)])\FF \Zeta(G)) \oplus (\Zeta(\FF G) \cap \Comm(\FF G)).
\end{equation*}
The last sum is direct as the first summand is contained in $\FF \Zeta(G)$.
By \cref{lem:augA} we finally obtain
\begin{equation}\label{eq:claim}
    \Theta(\FF G) = \Ideal(\FF[\Agemo_m(\Zeta(G) \cap \Agemo^*_n(G))])\FF \Zeta(G) \oplus (\Zeta(\FF G) \cap \Comm(\FF G)).
\end{equation}
Using \cref{lem:center}\ref{item:FZG} we have
\begin{equation*}
    \Zeta(\FF G)/\Theta(\FF G)
    = \frac{\FF \Zeta(G) \oplus (\Zeta(\FF G) \cap \Comm(\FF G))}{\Ideal(\FF[\Agemo_m(\Zeta(G) \cap \Agemo^*_n(G))])\FF \Zeta(G) \oplus (\Zeta(\FF G) \cap \Comm(\FF G))}
\end{equation*}
and as the first summand of \eqref{eq:claim} is an ideal of $\FF \Zeta(G)$, we get
\begin{equation*}
    \Zeta(\FF G)/\Theta(\FF G)
    \cong \FF \Zeta(G)/\Ideal(\FF[\Agemo_m(\Zeta(G) \cap \Agemo^*_n(G))])\FF \Zeta(G).
\end{equation*}
Thanks to \cref{lem:quotient}, the last quotient is isomorphic to
\begin{equation*}
    \FF[\Zeta(G)/\Agemo_m(\Zeta(G) \cap \Agemo^*_n(G))].
\end{equation*}
The modular isomorphism problem being positively solved for abelian groups, the isomorphism type of $\Zeta(G)/\Agemo_m(\Zeta(G) \cap \Agemo^*_n(G))$ is an invariant of $\FF G$ and, $m$ having been chosen arbitrarily, so is the sequence $(|\Agemo_k(\Zeta(G) \cap \Agemo^*_n(G))|)_{k \geq 0}$.

\ref{it:ZA*2} This is a special case of the last claim, where $m = 0$.
	
\ref{it:ZA*1}	The sequence $(|\Agemo_k(\Zeta(G) \cap \Agemo^*_n(G))|)_{k \geq 0}$ is an invariant of $\FF G$ and so \cref{prop:type} yields the isomorphism type of $\Zeta(G) \cap \Agemo^*_n(G)$ is an invariant of $\FF G$.
\end{proof}

The combination of \cref{thm:CO*} with \cref{thm:ZA*} is the same as \cref{main:invariants}, which is therefore now proven. We remark that the proof of \cref{main:invariants} is inspired by the argument for the invariance of $\Zeta(G) \cap \Comm(G)$ and $\Zeta(G)/\Zeta(G) \cap \Comm(G)$ by Sandling~\cite[Theorem~6.11]{San85}.

\begin{cor}\label{cor:invariants}
	The isomorphism types of the following are invariants of $\FF G$:
	\begin{multicols}{2} \begin{itemize}
		\item $G/\Comm(G)$,
		\item $G/\Comm(G)\Soc(G)$,
		\item $G/\Comm(G)\Zeta(G)$,
		\item $\Comm(G)\Soc(G)/\Comm(G)$,
		\item $\Comm(G)\Zeta(G)/\Comm(G)$,
		\item $G/\Frat(G)\Soc(G)$,
		\columnbreak
		\item $\Zeta(G)$,
		\item $\Zeta(G) \cap \Frat(G)$,
		\item $\Zeta(G) \cap \Comm(G)$,
		\item $\Zeta(G)/\Zeta(G) \cap \Frat(G)$,
		\item $\Zeta(G)/\Zeta(G) \cap \Comm(G)$,
		\item $\Soc(G) \cap \Frat(G)$.
	\end{itemize} \end{multicols}
\end{cor}

\begin{proof}
	Taking $n$ large enough ($n \geq \log_p |G|$, for example) in \cref{main:invariants} yields the following four invariants:
	\begin{equation*}
		G/\Comm(G)\Zeta(G), \quad \Comm(G)\Zeta(G)/\Comm(G), \quad \Zeta(G) \cap \Comm(G), \quad \Zeta(G)/\Zeta(G) \cap \Comm(G).
	\end{equation*}
	Setting $n = 0$ in \ref{it:invariants1} and \ref{it:invariants3} of \cref{main:invariants} yields two classic invariants:
	\begin{equation*}
		G/\Comm(G), \quad \Zeta(G).
	\end{equation*}
	Setting $n = 1$ in \cref{main:invariants} yields four new invariants
	\begin{equation*}
		G/\Comm(G)\Soc(G), \quad \Comm(G)\Soc(G)/\Comm(G), \quad \Zeta(G) \cap \Frat(G), \quad  \Zeta(G)/\Zeta(G) \cap \Frat(G).
	\end{equation*}
	The remaining two new invariants are obtained from $G/\Comm(G)\Soc(G)$ and $\Zeta(G) \cap \Frat(G)$ by observing that
	\begin{align*}
		G/\Frat(G)\Soc(G) &\cong \frac{G/\Comm(G)\Soc(G)}{\Agemo_1(G)\Comm(G)\Soc(G)/\Comm(G)\Soc(G)} = \frac{G/\Comm(G)\Soc(G)}{\Agemo_1(G/\Comm(G)\Soc(G))}, \\
		\Soc(G) \cap \Frat(G) &= \Omega_1(\Zeta(G)) \cap \Frat(G) = \Omega_1(\Zeta(G) \cap \Frat(G)).
		\qedhere
	\end{align*}
\end{proof}

We remark that  the invariants $\Zeta(G), G/\Comm(G) , \Zeta(G) \cap \Comm(G), \Zeta(G)/\Zeta(G) \cap \Comm(G)$ from \cref{cor:invariants} can be found in Sandling's survey paper~\cite[Theorems 6.11, 6.12, 6.7]{San85}.

\begin{rmk}
If $G$ has class $2$, one can see the isomorphism type of $G/\Zeta(G)$ as an invariant of $\FF G$ (cf.~\cite[Theorem~6.23]{San85}) from one of the above invariants, namely $G/\Comm(G)\Zeta(G) \cong (G/\Zeta(G))/\Comm(G/\Zeta(G))$.
\end{rmk}

\section{Elementary decompositions}
\label{sec:decomp}

In this section we present an elementary group-theoretic result on the decomposition of a finite $p$-group as a direct product of two subgroups, one of which is elementary abelian and maximal with these properties. 
This turns out to be fundamental for the proof of \cref{main:reduction}.

\begin{dfn}
	\label{dfn:decomp}
	A decomposition
	\begin{equation}
		G = T \times U
	\end{equation}
	of $G$ into subgroups $T$ and $U$ is said to be \emph{elementary} if the following are satisfied:
	\begin{enumerate}
		\item\label{item:T} $T \cap \Frat(G) = 1$ and $|T| = | \Soc(G) : \Soc(G) \cap \Frat(G) |$; and
		\item\label{item:U} $\Soc(G)U = G$ and $|U| = |G|/| \Soc(G) : \Soc(G) \cap \Frat(G) |$.
	\end{enumerate}
\end{dfn}

It can be proven that, given a decomposition $G = T \times U$, Conditions~\ref{item:T} and \ref{item:U} are equivalent.
Moreover, as defined in \cref{dfn:decomp}, the subgroup $T$ of $G$ is maximal with the property that $T$ is an elementary abelian direct factor of $G$.

\begin{lem}
	\label{lem:exists}
	An elementary decomposition $G = T \times U$ with $T$ elementary abelian always exists and
	the subgroups $T$ and $U$ are unique up to isomorphism.
	Furthermore, the following equalities hold:
	\[
		\Soc(G) = T \times \Soc(U)
		\textup{ and }
		\Frat(G) = \Frat(U).
	\]
\end{lem}

\begin{proof}
	Since $\Soc(G)$ is elementary abelian, the subgroup $\Soc(G) \cap \Frat(G)$ has a complement $T$ in $\Soc(G)$.
	We take such $T$ and remark that $T \cap \Frat(G) = 1$ and $\Soc(G) = T(\Soc(G) \cap \Frat(G))$.
	Since $G/\Frat(G)$ is also elementary abelian, the subgroup $\Soc(G)\Frat(G)/\Frat(G)$ has a complement in $G/\Frat(G)$;
	we take it to be of the form $U/\Frat(G)$ with $U$ containing $\Frat(G)$.
	It follows that $\Soc(G)\Frat(G) \cap U = \Frat(G)$ and, the elements of $\Frat(G)$ being non-generators, that $G = \Soc(G)U$;
	see \cref{fig:elem}.
	\begin{figure}[htpb]
		\includegraphics{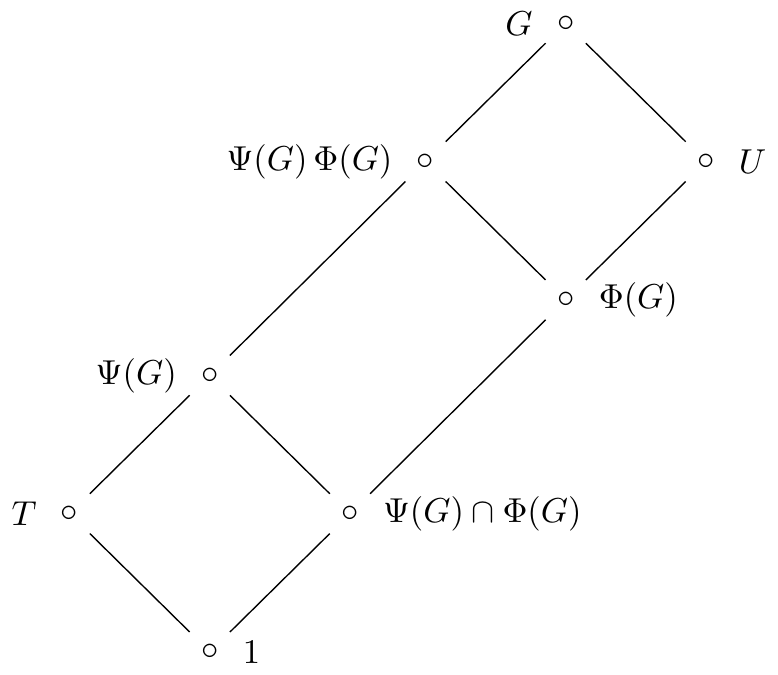}
		\caption{Subgroups involved in an elementary decomposition.}
		\label{fig:elem}
	\end{figure}
	
	\noindent
	It is now easy to deduce
	that $G = TU$ and $T \cap U = 1$.
	Since $T$ is central, we obtain $G = T \times U$ and this decomposition is elementary.
	The uniqueness of $T$ and $U$, up to isomorphism,  follows from the Krull-Remak-Schmidt theorem.
	In conclusion, since $T$ is an elementary abelian direct factor of $G$, we have $\Soc(G) = T \times \Soc(U)$ and $\Frat(G) = \Frat(U)$.
\end{proof}

\begin{rmk}
	\label{rmk:p-stem}
	Let $G = T \times U$ be an elementary decomposition of $G$ with $T$ elementary abelian.
	Note that there are natural isomorphisms $G/\Soc(G) \to U/\Soc(U)$ and $\Frat(G) \to \Frat(U)$.
	Furthermore, these are compatible with $p$-th power maps $G/\Soc(G) \to \Frat(G)$ and $U/\Soc(U) \to \Frat(U)$, i.e.\ the following diagram commutes.
	\begin{equation*} \begin{tikzcd}
        G/\Soc(G) \arrow[r] \arrow[d]
        &
        U/\Soc(U) \arrow[d]
        \\
        \Frat(G) \arrow[r]
        &
        \Frat(U)
    \end{tikzcd} \end{equation*}
	Roughly speaking, $G$ and $U$ have the same power structure.
\end{rmk}

\section{A reduction theorem}\label{sec:reduction}

In this section we prove that elementary abelian direct factors can be disregarded in the study of the modular isomorphism problem. More precisely, we prove the following generalization of \cref{main:reduction} from the Introduction.

\begin{thm}\label{thm:reduction}
	Let $G = T \times U$ and $H = S \times V$ be elementary decompositions of $G$ and $H$, with $T$ and $S$ elementary abelian.
	Then the following are equivalent:
	\begin{enumerate}
	\item\label{it:reduction1} $\FF G\cong \FF H$.
	\item\label{it:reduction2} $\FF T \cong \FF S \textup{ and } \FF U \cong \FF V$.
	\end{enumerate}
\end{thm}

Until the end of the present section, we write $\II=\Ideal(\FF G)$ for the augmentation ideal of $\FF G$.

\begin{lem}
	\label{lem:filt}
	Let $K$ and $L$ be normal  subgroups of $G$ such that $G$ equals the internal direct product $K\times L$ of $K$ and $L$.
	Let, moreover, $n$ be a positive integer.
	Then the following equality holds:
	\begin{equation*}
		\II^n = \Ideal(\FF K)\II^{n - 1} \oplus \Ideal(\FF L)^n.
	\end{equation*}
\end{lem}
\begin{proof}
	We work by induction on $n$. Assume first that $n = 1$. Then, for each $k \in K$ and $l \in L$, one has $kl - 1 = (k - 1)l + (l - 1)$: this  proves
	$\II = \Ideal(\FF K)\FF G + \Ideal(\FF L)$. The last sum is direct thanks to \cref{lem:quotient}.

	Assume now that $n > 1$ and that $\II^{n-1} = \Ideal(\FF K)\II^{n - 2} \oplus \Ideal(\FF L)^{n-1}$.
	Then, the inclusion
	$\II^n \supseteq \Ideal(\FF K) \II^{n - 1} \oplus \Ideal(\FF L)^n$
	immediately follows from the definition of $\II$, while the opposite inclusion is obtained from combining $\II^n = \II^{n - 1}\II$, the induction hypothesis, and the case $n=1$.
\end{proof}

\begin{lem}\label{lem:augS}
	The following hold:
	\begin{enumerate}
		\item\label{item:Frat} $\Ideal(\FF \Frat(G))\FF G \subseteq \II^2$.
		\item\label{item:Soc} $\Ideal(\FF \Soc(G))\FF G + \II^2 = \Omega_1(\Zeta(\FF G)) + \II^2$.
	\end{enumerate}
\end{lem}
\begin{proof}
	\ref{item:Frat} This follows from combining \cref{lem:commutative,lem:product,lem:dream}.

	\ref{item:Soc} It is easy to show that $\Ideal(\FF \Soc(G))\FF G + \II^2 \subseteq \Omega_1(\Zeta(\FF G)) + \II^2$.
	To show the converse, it suffices to prove that $\Omega_1(\Zeta(\FF G)) \subseteq \Ideal(\FF \Soc(G))\FF G + \II^2$.
	To this end, take an element $x \in \Omega_1(\Zeta(\FF G))$ and, thanks to \cref{lem:center}\ref{item:Z}, write
	\begin{equation*}
		x = \sum_{z \in \Zeta(G)} \alpha_z z + \sum_{\substack{\kappa \in \cc(G)\\ |\kappa| \neq 1}} \beta_\kappa \hat\kappa
	\end{equation*}
for scalars $\alpha_z$, $\beta_\kappa \in \FF$.
By \cref{lem:center}\ref{item:KJ2} and \ref{item:ZK}, the second summand on the right-hand side  belongs to $\II^2$; we will prove that the first summand belongs to $\Ideal(\FF \Soc(G))\FF G$.
Now,	 from $x^p = 0$ and \cref{lem:center}\ref{item:ZK}, we derive that
\[
\sum_{z \in \Zeta(G)} \alpha_z z \in \Omega_1(\FF \Zeta(G)).
\]
Taking $m=0$ and $n=1$ in	\cref{lem:powers}, with $\Zeta(G)$ in the role of $G$,  yields that $\Omega_1(\FF \Zeta(G))\subseteq\Ideal(\FF \Soc(G)) \FF G$ and thus the claim.
\end{proof}

\begin{lem}
	\label{lem:augT}
	Let $G = T \times U$ be an elementary decomposition of $G$ with $T$ elementary abelian. Let, moreover, $n$ be a positive integer.
	Define $I = \Ideal(\FF T)\FF G$.
	Then the following hold:
\begin{enumerate}
\item\label{it:augT1} $\codim I = |G|/|\Soc(G) : \Soc(G) \cap \Frat(G)|.$
\item\label{it:augT2} $I \II^{n - 1} + \II^{n + 1} = \Ideal(\FF \Soc(G)) \II^{n - 1} + \II^{n + 1}$. 
\end{enumerate}	
\end{lem}

\begin{proof}
As \ref{it:augT1} follows immediately from \cref{lem:quotient}, we show \ref{it:augT2}.
From the fact that $\Soc(U) \subseteq \Frat(U)$ and \cref{lem:exists}, we derive
	\begin{align*}
		I + \II^2
		&= \Ideal(\FF T)\FF G + \II^2 \\
		&= \Ideal(\FF T)\FF G + \Ideal(\FF \Soc(U))\FF G + \II^2 \\
		&= \Ideal(\FF \Soc(G))\FF G + \II^2.
	\end{align*}
It now follows from \cref{lem:product,lem:augS} that 
	\begin{align*}
		I \II^{n - 1} + \II^{n + 1}
		& = (I + \II^2)\II^{n - 1}
		= (\Ideal(\FF \Soc(G))\FF G + \II^2) \II^{n - 1}\\
		&= \Ideal(\FF \Soc(G)) \II^{n - 1} + \II^{n + 1}.
		\qedhere
	\end{align*}
\end{proof}

\begin{lem}
	\label{lem:char}
	Let $G = T \times U$ be an elementary decomposition of $G$ with $T$ elementary abelian.
	Let $I$ be an ideal of $\FF G$ satisfying
	\begin{enumerate}
\item\label{it:char1} $\codim I = |G|/|\Soc(G) : \Soc(G) \cap \Frat(G)|$ and
\item\label{it:char2} for each integer $n>0$, one has $I \II^{n - 1} + \II^{n + 1} = \Ideal(\FF \Soc(G)) \II^{n - 1} + \II^{n + 1}$. 
\end{enumerate}	
	Then one has $\FF G/I \cong \FF U$.
\end{lem}
\begin{proof}
In order to prove the lemma, we will show that $\FF G = I \oplus \FF U$. To do this, we first prove that, for each positive integer $n$, one has
	\begin{equation}
		\label{eq:filt}
		\II^n = I \II^{n - 1} + \Ideal(\FF U)^n.
	\end{equation}
Let $c$ be minimal such that $\Delta^c=0$. Then, for each $n>c$, the equality in \eqref{eq:filt} clearly holds. Assume now that $0<n\leq c$: we prove \eqref{eq:filt} by induction on $c-n$. 		
If $n=c$, we are done because $I$ is contained in the unique maximal ideal $\II$. Suppose now that $n<c$ and that $\II^{n + 1} = I \II^n + \Ideal(\FF U)^{n + 1}$.
	From \cref{lem:filt,lem:augT} and \ref{it:char2} we then obtain
		\begin{align*}
		\II^n
		&= \II^n + \II^{n + 1} \\
		&= \Ideal(\FF T) \II^{n - 1} + \II^{n + 1} + \Ideal(\FF U)^n \\
		&= \Ideal(\FF \Soc(G)) \II^{n - 1} + \II^{n + 1} + \Ideal(\FF U)^n \\
		&= I \II^{n - 1} + \II^{n + 1} + \Ideal(\FF U)^n \\
		&= I \II^{n - 1} + I \II^n + \Ideal(\FF U)^{n + 1} + \Ideal(\FF U)^n \\
		&= I \II^{n - 1} + \Ideal(\FF U)^n.
	\end{align*}
	This completes the proof of \eqref{eq:filt}. In particular, from $n=1$, we obtain that $\II = I + \Ideal(\FF U)$ and thus $\FF G = I + \FF U$.
The last equality, together with the condition on the codimension of $I$, yields that $I\cap \FF U=0$ and hence we have $\FF G = I \oplus \FF U$.
\end{proof}

We conclude the current section by proving \Cref{thm:reduction}.
\begin{proof}[Proof of \cref{thm:reduction}]
Let $G = T \times U$ and $H = S \times V$ be elementary decompositions with $T$ and $S$ elementary abelian.

For the proof of $\ref{it:reduction2}\Rightarrow\ref{it:reduction1}$, assume that $\FF T \cong \FF S$ and $\FF U \cong \FF V$.
It follows from \cite[Chapter 1, Lemma 3.4]{Pas77} that
\[
\FF G =\FF [T \times U] \cong \FF T \otimes_\FF \FF U \cong  \FF S \otimes_\FF \FF V \cong \FF [S \times V]=\FF H.
\]

For the proof of $\ref{it:reduction1}\Rightarrow \ref{it:reduction2}$, let $\varphi:\FF G\rightarrow \FF H$ be an algebra isomorphism. Note that $T \cong \Soc(G)/\Soc(G) \cap \Frat(G)$ and $S \cong \Soc(H)/\Soc(H) \cap \Frat(H)$, so the combination of \cref{thm:WS} and \cref{cor:invariants} yields that $T \cong S$. Let $I$
be an ideal of $FG$ satisfying the conditions from \cref{lem:char}; such an ideal exists as a consequence of \cref{lem:augT}.
Now, as a consequence of \cref{lem:augS}\ref{item:Soc}, the ideal $\Ideal(\FF \Soc(G))\FF G + \II^2$ is canonical in $\FF G$ and, thanks to \cref{cor:invariants}, the quantity $|G|/|\Soc(G) : \Soc(G) \cap \Frat(G)|$ is an invariant of $\FF G$. It follows in particular that the ideal $\varphi(I)$ of $\FF H$ also satisfy the conditions from \cref{lem:char} and thus we have that
\[\FF U \cong \FF G/I \cong \FF H/\varphi(I) \cong \FF V.\]
The proof is complete.
\end{proof}

\section{Applications}
\label{sec:appl}

This section is devoted to particular instances in which our main results can be applied to solve the modular isomorphism problem. Since our aim is to work independently of the choice of the field, in \cref{sec:appl} we avoid the use of invariants which are only known to hold over the prime field. We remark that, even if working over $\F_p$ would shorten some arguments, the results here presented would still be new.  

\subsection{Infinitely to finitely many}
In the current section we positively solve the modular isomorphism problem for the following infinite family of $2$-groups of class $3$ by first reducing it to a finite analysis using our \cref{main:reduction} and subsequently applying old and new invariants, cf.\ \cref{main:invariants}.

\begin{thm}\label{thm:1st}
Assume $G$ has class $3$ and satisfies $|G:\Zeta(G)| = |\Frat(G)| = 8$. Then the modular isomorphism problem has a positive answer for $G$.
\end{thm}

\begin{lem}
	\label{lem:ubd}
	Assume $G$ has class $3$ and satisfies $|G : \Zeta(G)| = |\Frat(G)| = p^3$. Then one has
	 $|G|/|\Soc(G) : \Soc(G) \cap \Frat(G)| < p^8$.
\end{lem}
\begin{proof}
	First, we prove the following inequality, which in fact holds for any $p$-group:
	\begin{equation}
		\label{eq:ubd}
		|G|/|\Soc(G) : \Soc(G) \cap \Frat(G)|
		\le |G : \Zeta(G)\Frat(G)| \cdot |\Frat(G)|^2.
	\end{equation}
	Consider the chain of subgroups
	\begin{equation*}
		G \supseteq \Zeta(G)\Frat(G) \supseteq \Zeta(G) \supseteq \Soc(G) \supseteq \Soc(G) \cap \Frat(G) \supseteq 1
	\end{equation*}
	and
	use it to decompose $|G|/|\Soc(G) : \Soc(G) \cap \Frat(G)|$ as
	\begin{align*}
		& |G|/|\Soc(G) : \Soc(G) \cap \Frat(G)| \\
		& \qquad = |G : \Zeta(G)\Frat(G)| \cdot |\Zeta(G)\Frat(G) : \Zeta(G)| \cdot |\Zeta(G) : \Soc(G)| \cdot |\Soc(G) \cap \Frat(G)|.
	\end{align*}
	One can see that the product of the second and fourth factors is bounded as
	\begin{align*}
		& |\Zeta(G)\Frat(G) : \Zeta(G)| \cdot |\Soc(G) \cap \Frat(G)| \\
		& \qquad = |\Frat(G) : \Zeta(G) \cap \Frat(G)| \cdot |\Soc(G) \cap \Frat(G)| \\
		& \qquad \le |\Frat(G) : \Soc(G) \cap \Frat(G)| \cdot |\Soc(G) \cap \Frat(G)|
		= |\Frat(G)|,
	\end{align*}
	and the third factor is bounded as
	\begin{align*}
		|\Zeta(G) : \Soc(G)|
		= |\Zeta(G) : \Omega_1(\Zeta(G))|
		= |\Agemo_1(\Zeta(G))|
		\le |\Agemo_1(G)|
		\le |\Frat(G)|.
	\end{align*}
	This completes the proof of \eqref{eq:ubd}.

	Now it follows from the assumption that the central quotient $G/\Zeta(G)$ is a non-abelian group of order $p^3$ and thus its minimal number of generators is two.
	Since
	\begin{equation*}
		G/\Zeta(G)\Frat(G) \cong \frac{G/\Zeta(G)}{\Zeta(G)\Frat(G)/\Zeta(G)} = \frac{G/\Zeta(G)}{\Frat(G/\Zeta(G))},
	\end{equation*}
	we obtain $|G : \Zeta(G)\Frat(G)| = p^2$.
	Furthermore, one can see that $\Agemo_1(\Zeta(G)) \neq \Frat(G)$ since $G$ has class $3$.
	Thus the inequality in \eqref{eq:ubd} is strict and we obtain
	\begin{equation*}
		|G|/|\Soc(G) : \Soc(G) \cap \Frat(G)| < p^2 \cdot (p^3)^2 = p^8.
		\qedhere
	\end{equation*}
\end{proof}

In the following result and later in this paper,  we denote by $\SG(n,m)$ the $m$-th group of order $n$ in the Small Groups Library \cite{BEOG20} of GAP \cite{GAP21}.

\begin{lem}\label{lem:finite}
Assume $G$ has class $3$ and satisfies $|G:\Zeta(G)| = |\Frat(G)| = 8$.
Let $G = T \times U$ be an elementary decomposition with $T$ elementary abelian.
Then $U$ is isomorphic to one of the following groups.
\begin{alignat*}{3}
	U_1    &= \SG(32,9),     & \qquad U_2    &= \SG(32, 10), & \qquad U_3    &= \SG(32,11),
	\\
	U_4    &= \SG(32,13),    & \qquad U_5    &= \SG(32,14),  & \qquad U_6    &= \SG(32,15),
	\\
	U_7    &= \SG(64,97),    & \qquad U_8    &= \SG(64,108), & \qquad U_9    &= \SG(64,118),
	\\
	U_{10} &= \SG(64,119),   & \qquad U_{11} &= \SG(64,120), & \qquad U_{12} &= \SG(64,124),
	\\
	U_{13} &= \SG(128,1671). &               &               &               &
\end{alignat*}

\end{lem}
\begin{proof}
The order of $U$ is bounded by $2^7$, as consequence of \cref{dfn:decomp}\ref{item:U} and \cref{lem:ubd}.  A quick search with GAP \cite{GAP21}, leveraging on \cref{lem:exists}, yields the claim.
\end{proof}

To finish the proof of
\cref{thm:1st}, we collect a number of invariants from the literature known to be valid over any field of characteristic $p$. In the following proposition, we say that a conjugacy class $\kappa$ \emph{consists of $p^n$-th powers} if $\kappa$ is the conjugacy class of  $g^{p^n}$ for some $g \in G$.

\begin{prop}\label{prop:known}
Let $n$ be a non-negative integer.	The following are invariants of $\FF G$:
	\begin{enumerate}
		\item\label{item:Ku} the number $k_n(G)$ of conjugacy classes of $G$ consisting of $p^n$-th powers,
		\item\label{item:Qu} the number $a_n(G)$ of conjugacy classes of maximal elementary abelian subgroups of $G$ of rank $n$,
		\item\label{item:H*} the number $e(G)$ of elements in a minimal set of generators of the mod-$p$ cohomology ring $\cohomology(G, \F_p)$.
	\end{enumerate}
\end{prop}
\begin{proof}
	The first invariant $k_n(G)$ is obtained from an observation of K\"ulshammer \cite[Section~1]{Kue82}, see also \cite[Proposition~1]{San96} or \cite[Section~2.2]{HS06}.
	The second invariant $a_n(G)$ is derived from Quillen's Stratification Theorem \cite[Corollary~8.3.3]{Eve91}, to be also found in \cite[Section~2.5]{HS06}.
As for the last point, the cohomology ring $\cohomology(G, \FF) = \Ext^*_{\FF G}(\FF, \FF)$ of $G$ can be computed directly from $\FF G$,
	and hence the number of elements in a minimal set of generators of $\cohomology(G, \FF)$ as a graded algebra is also an invariant of $\FF G$; cf.\ \cite[p.~315]{CTVEZ03}.
	Since
	\begin{equation*}
		\cohomology(G, \FF) \cong \cohomology(G, \F_p) \otimes_{\F_p} \FF
	\end{equation*}
	by \cite[Section 3.4]{Eve91}, the number $e(G)$ is an invariant of $\FF G$.
\end{proof}

	We conclude the current section by proving \cref{thm:1st}. 
\begin{proof}[Proof of \cref{thm:1st}]
	Assume $\FF G \cong \FF H$.
	First, we show that $H$ also has class $3$ and satisfies $|H : \Zeta(H)| = |\Frat(H)| = 8$.
	Clearly $|G| = |H|$, and $H$ satisfies the last two conditions by \cref{thm:WS,prop:Frat}.
	Since $G$ has class $3$, we have $\Zeta(G) \not\supseteq \Comm(G)$ and $|\Comm(G)| > |\Zeta(G) \cap \Comm(G)|$.
	The same inequality holds for $H$ by \cref{cor:invariants} and $H$ has class $3$ as $|H : \Zeta(H)| = 8$; see also \cite[Theorem~2]{BK07}.

	Without loss of generality and thanks to \cref{thm:reduction}, we assume that $G$ and $H$ do not have elementary abelian direct factors. Then, as a consequence of
	\cref{lem:finite}, we are only concerned with the determination of $13$ groups from their modular group algebras.
	We do so with the aid of the following invariants of modular group algebras, listed in \cref{thm:WS,prop:known}:
	\begin{itemize}
		\item the number $k_n(G)$ of conjugacy classes of $G$ consisting $p^n$-th powers,
		\item the number $a_n(G)$ of conjugacy classes of maximal elementary abelian subgroups of $G$ of rank $n$,
		\item the number $e(G)$ of elements in a minimal set of generators of $\cohomology(G, \F_2)$ as a graded algebra,
		\item the order of the socle $\Soc(G)$.
	\end{itemize}
	\newcommand{\headerlength}{10ex}
	\newcommand{\columnlength}{6ex}
	\begin{table}[h]
		\centering
		\begin{tabular}{p{\headerlength}p{\columnlength}p{\columnlength}p{\columnlength}p{\columnlength}p{\columnlength}p{\columnlength}}
			\toprule
			$G$      & $U_1$ & $U_2$ & $U_3$ & $U_4$ & $U_5$ & $U_6$ \\
			\midrule
			$k_1(G)$ & $4$   & $4$   & $5$   & $5$   & $4$   & $5$ \\
			$a_2(G)$ & $0$   & $1$   & $2$   & $1$   & $1$   & $1$ \\
			$e(G)$   & $5$   & $6$   & $6$   & $4$   & $4$   & $5$ \\
			\bottomrule \\
		\end{tabular}
		\caption{Some invariants for the groups from \cref{lem:finite} of order $32$.}
		\label{tbl:32}
	\end{table}
	\begin{table}[h]
		\centering
		\begin{tabular}{p{\headerlength}p{\columnlength}p{\columnlength}p{\columnlength}p{\columnlength}p{\columnlength}p{\columnlength}}
			\toprule
			$G$         & $U_7$ & $U_8$ & $U_9$ & $U_{10}$ & $U_{11}$ & $U_{12}$ \\
			\midrule
			
			$k_1(G)$    & $5$   & $5$   & $6$   & $6$      & $6$      & $6$ \\
			$a_3(G)$    & $2$   & $1$   & $2$   & $1$      & $0$      & $0$ \\
			$|\Soc(G)|$ & $4$   & $4$   & $4$   & $4$      & $4$      & $2$ \\
			\bottomrule \\
		\end{tabular}
		\caption{Some invariants for the groups from \cref{lem:finite} of order $64$.}
		\label{tbl:64}
	\end{table}
	The numbers $k_n(G)$ and $a_n(G)$ can be calculated using \texttt{MIPConjugacyClassInfo} and \texttt{SubgroupsInfo} in ModIsomExt~\cite{MM20}, for example.
	See \cite{GK15} or \cite[Appendix~D]{CTVEZ03} for the number $e(G)$.
	It follows from \cref{tbl:32,tbl:64} that $G \cong H$.
\end{proof}

We remark that our proof of \cref{thm:1st} could be slightly shortened with the aid of \cite[Lemma~3.7]{NS18}, where the modular isomorphism problem is shown to have a positive solution for groups of order $32$ over any field of characteristic $2$. Since the techniques are similar and we also deal with groups of order $64$, we present here the proof of \cref{thm:1st} in its entirety.

\begin{example}
As noted in \cite{HS06}, the modular group algebras of the groups $\SG(64, 97)$ and $\SG(64, 101)$ as well as $\SG(64, 108)$ and $\SG(64,110)$ cannot be distinguished  by known group-theoretic invariants. Indeed, in \cite[Sections~4.1, 4.2]{HS06} ring-theoretic techniques are used to solve the modular isomorphism problem for these groups. Our new invariants from Theorem~\ref{thm:CO*} or \ref{thm:ZA*} can distinguish the modular group algebras of those groups and in fact they are covered by \cref{thm:1st} as it is evident from \cref{lem:finite}. See also \cref{cor:2nd}.
\end{example}

\subsection{Generalizing maximal class}
We refer to a group as {\em dihedral} if it is a finite $2$-group that is generated by precisely two elements of order $2$. In particular, the Klein four group is dihedral.
In this section, we provide a positive answer to the modular isomorphism problem for the following class of finite $2$-groups.
\begin{thm}\label{thm:2nd}
Assume that $G$ has cyclic center $\Zeta(G)$ and dihedral central quotient $G/\Zeta(G)$. Then the modular isomorphism problem has a positive answer for $G$.
\end{thm}

As we will see from its proof, the last theorem is an instance of the successful application of our new invariants, more specifically $G/\Comm(G)\Zeta(G)$; see \cref{cor:invariants}.
Moreover, we remark that the class from \cref{thm:2nd} contains the dihedral group $\Dih_{2^{1 + n}}$  and the generalized quaternion group $\Quat_{2^{1 + n}}$, whenever $n \ge 2$, and the semidihedral group $\Semi_{2^{1 + n}}$, for $n \ge 3$.  These have presentations
\begin{alignat*}{5}
	\Dih_{2^{1+n}}
	&= \langle\, a, b, c \mid
	a^2 &= 1, \ \
	b^2 &= 1, \ \
	(ab)^{2^{n - 1}} = c,& \    \
	c^2 &= [a, c] = [b, c] = 1 \,\rangle, \\
	\Quat_{2^{1+n}}
	&= \langle\, a, b, c \mid
	a^2 &= c,\ \
	b^2 &= c, \ \
	(ab)^{2^{n - 1}} = c^{1 + 2^{n - 1}},& \ \
	c^2 &= [a, c] = [b, c] = 1 \,\rangle, \\
	\Semi_{2^{1+n}}
	&= \langle\, a, b, c \mid
	a^2 &= 1, \ \
	b^2 &= c, \ \
	(ab)^{2^{n - 1}} = c^{1 + 2^{n - 2}},&  \ \
	c^2 &= [a, c] = [b, c] = 1 \,\rangle.
\end{alignat*}
More generally, we define the following three families of $2$-groups for each $m \ge 1$ and $n \ge 2$:
\begin{alignat*}{5}
	\Dih_{2^{m|n}}
	&= \langle\, a, b, c \mid
	a^2 &= 1, \ \
	b^2 &= 1, \ \
	(ab)^{2^{n - 1}} = c^{2^{m - 1}},& \    \
	c^{2^m} &= [a, c] = [b, c] = 1 \,\rangle, \\
	\Quat_{2^{m|n}}
	&= \langle\, a, b, c \mid
	a^2 &= c,\ \
	b^2 &= c, \ \
	(ab)^{2^{n - 1}} = c^{2^{m - 1} + 2^{n - 1}},& \ \
	c^{2^m} &= [a, c] = [b, c] = 1 \,\rangle, \\
	\Semi_{2^{m|n}}
	&= \langle\, a, b, c \mid
	a^2 &= 1, \ \
	b^2 &= c, \ \
	(ab)^{2^{n - 1}} = c^{2^{m - 1} + 2^{n - 2}},&  \ \
	c^{2^m} &= [a, c] = [b, c] = 1 \,\rangle.
\end{alignat*}
The last groups are central extensions of a dihedral group of order $2^n$
by a cyclic group of order $2^m$.
Note that there are exceptional isomorphisms
\begin{equation}
	\label{eq:exceptions}
	\Semi_{2^{m|2}} \cong
	\begin{cases}
		\Dih_{2^{1|2}} & \textup{if } m = 1, \\
		\Quat_{2^{m|2}} & \textup{if } m > 1,
	\end{cases}
\end{equation}
given by
\begin{alignat*}{3}
	& \Semi_{2^{1|2}} \to \Dih_{2^{1|2}},  & \qquad a &\mapsto a,                   & \qquad b &\mapsto ab, \\
\intertext{and}
	& \Semi_{2^{m|2}} \to \Quat_{2^{m|2}}, & \qquad a &\mapsto abc^{2^{m - 2} - 1}, & \qquad b &\mapsto b.
\end{alignat*}

The following theorem allows us to study the family from \cref{thm:2nd} in terms of the just-defined groups.

\begin{thm}\label{thm:trichotomy}
Assume that $G$ has cyclic center $\Zeta(G)$ of order $2^m$ and dihedral central quotient $G/\Zeta(G)$ of order $2^n$.
Then $G$ is isomorphic to one of $\Dih_{2^{m|n}}$, $\Quat_{2^{m|n}}$ or $\Semi_{2^{m|n}}$.
\end{thm}

\begin{proof}
	Observe that $m\geq 1$, because finite non-trivial $p$-groups have non-trivial centers, and $n\geq 2$, because $G/\Zeta(G)$ cannot be a non-trivial cyclic group.
	It follows from the assumptions that there are elements $x$, $y$, $z$ generating $G$ and integers $0 \le \xi, \eta, \zeta < 2^m$ with
	\begin{equation*}
		x^2 = z^\xi,\ y^2 = z^\eta,\ (xy)^{2^{n - 1}} = z^\zeta,\ z^{2^m} = [x, z] = [y, z] = 1.
	\end{equation*}
	First, we shall replace the generators $x,y,z$ with new generators $a,b,c$ that are relatable via simpler parameters.
	Let $q_\xi$ and $r_\xi$ be the quotient and remainder of $\xi$ divided by $2$ so that
	\[ \xi = 2q_\xi + r_\xi \textup{ and } 0\leq r_\xi<2.\]
	Similarly, define $q_\eta$ and $r_\eta$.
	Replace the generators $x$, $y$, $z$ with
	\[ a = xz^{-q_\xi} ,\quad b = yz^{-q_\eta} ,\quad c = z \]
	and note that $a,b,c$ satisfy relations of the same form
	\begin{equation*}
		a^2 = c^\alpha,\quad b^2 = c^\beta,\quad (ab)^{2^{n - 1}} = c^\gamma,\quad c^{2^m} = [a, c] = [b, c] = 1
	\end{equation*}
	where the parameters' values range as $0 \le \alpha, \beta < 2$ and $0 \le \gamma < 2^m$;
	we assume, without loss of generality, that $\alpha \le \beta$.
	In particular, there are only three possibilities for $(\alpha, \beta)$.
	Suppose $(\alpha, \beta) = (0, 1)$, namely, $a^2 = 1$ and $b^2 = c$.
	Then
	\begin{align*}
		c^\gamma
		&= (ab)^{2^{n - 1}}
		= (ab)^{2^{n - 1}}a^2
		= a(ba)^{2^{n - 1}}a
		= a(abc^{-1})^{-2^{n - 1}}a \\
		&= a(ab)^{-2^{n - 1}}ac^{2^{n - 1}}
		= ac^{-\gamma}ac^{2^{n - 1}}
		= a^2c^{-\gamma + 2^{n - 1}}
		= c^{-\gamma + 2^{n - 1}},
	\end{align*}
	which yields $c^{2\gamma - 2^{n - 1}} = 1$.
	The element $c$ belonging to the center, we obtain that
	$2\gamma - 2^{n - 1} \equiv 0 \bmod 2^m$ and
	hence $\gamma \equiv 2^{m - 1} + 2^{n - 2} \bmod 2^m$.
	A similar argument can be applied, for $(\alpha, \beta) = (1, 1)$, to prove that
	$\gamma \equiv 2^{m - 1} + 2^{n - 1} \bmod 2^m$  and, for $(\alpha, \beta) = (0, 0)$, to show that
	$\gamma \equiv 2^{m - 1} \bmod 2^m$.
	This completes the proof.
\end{proof}

We will prove \cref{thm:2nd} following the next steps. 
We will first show that if $\FF G \cong \FF H$ and $G$ belongs to the family $\mathcal{X}$ of $2$-groups from \cref{thm:2nd}, then so does $H$.
Then we will use a number of invariants, to show that any two group bases associated to the same algebra possess the same invariants $(m,n)$; cf.\ \cref{thm:trichotomy}. To conclude, we will analyze different (incompatible) properties of the modular group algebras of  $\Dih_{2^{m|n}}$, $\Quat_{2^{m|n}}$, and $\Semi_{2^{m|n}}$ and leverage on \cref{thm:trichotomy} to positively solve the modular isomorphism problem for $\mathcal{X}$.

\begin{lem}
	\label{lem:cyclic}
	The center $\Zeta(G)$ is cyclic if and only if $| \Soc(G) | = p$.
\end{lem}

\begin{proof}
This follows from the fact that an abelian $p$-group that has a unique subgroup of order $p$ is cyclic.
\end{proof}

\begin{lem}
	\label{lem:capable}
	The quotient $G/\Zeta(G)$ is dihedral if and only if $| G : \Comm(G)\Zeta(G) | = 4$.
\end{lem}

\begin{proof}
We will make use of the following basic fact:
\begin{equation*}
	G/\Comm(G)\Zeta(G) \cong (G/\Zeta(G))/\Comm(G/\Zeta(G)).
\end{equation*}
If $G/\Zeta(G)$ is dihedral, then $| G : \Comm(G)\Zeta(G) | = 4$ follows immediately.
Suppose that $| G : \Comm(G)\Zeta(G) | = 4$.
If the central quotient $G/\Zeta(G)$ is abelian, then it cannot be cyclic and is thus isomorphic to the Klein four group $\Dih_4$.
We assume now that $G/\Zeta(G)$ is non-abelian.
By Taussky's trichotomy theorem~\cite[Kapitel~III, Satz~11.9]{Hup67}, the group $G/\Zeta(G)$ is isomorphic to the dihedral group $\Dih_{2^{1 + n}}$ or the generalized quaternion group $\Quat_{2^{1 + n}}$, for $n \ge 2$, or the semidihedral group $\Semi_{2^{1 + n}}$, for $n \ge 3$.
Moreover, \cite[Lemma~3.1]{HS64} ensures that $\Quat_{2^{1 + n}}$ and $\Semi_{2^{1 + n}}$ are incapable of being central quotients, so $G/\Zeta(G)$ must be dihedral.
\end{proof}

For the proof of \cref{thm:2nd} we will use dimension subgroups, but will avoid presenting here an extensive description of their properties. We refer to \cite[Section~1]{HS06} or \cite[Section~2.3]{MS21} for more details and references. For a positive integer $n$, the $n$-th \emph{dimension subgroup} $D_n(G)$ of $G$ is defined as
\[D_n(G) = G \cap (1 + \Ideal(\FF G)^n). \]
Jennings showed that the dimension subgroups form a $p$-restricted $N$-series, i.e. that for any pair of indices $n$ and $m$ one has
\begin{equation}\label{eq:Jen}
[D_n(G), D_m(G)] \subseteq D_{n+m}(G) \textup{ and } \Agemo_1(D_n(G)) \subseteq D_{np}(G).
\end{equation}
Moreover, the \emph{dimension subgroups series}, also called the \emph{Jennings series}, is the fastest descending series of subgroups in $G$ satisfying \eqref{eq:Jen}.
If $p = 2$, for instance,  the first few terms are
\begin{align*}
	D_1(G) = G, \quad D_2(G) = \Frat(G), \quad D_3(G) = [\Frat(G), G]\Agemo_1(\Frat(G)).
\end{align*}
Moreover, for every choice of $n$, the quotient $D_n(G)/D_{n+1}(G)$ is elementary abelian and can be viewed as a vector space over $\F_p$.
It was also proven by Jennings that, given $g\in G$, one has $g-1 \in \Ideal(\FF G)^n$ if and only if $g \in D_n(G)$ and that, if $g_1,\dotsc,g_d$ form a basis of $D_n(G)/D_{n+1}(G)$, then $g_1-1,\dotsc,g_d-1$ are linearly independent in $\Ideal(\FF G)^n/\Ideal(\FF G)^{n+1}$. Moreover, after subtracting $1$, any basis of $D_1(G)/D_2(G) = G/\Frat(G)$ forms a basis of $\Ideal(\FF G)/\Ideal(\FF G)^2$.

In what follows we will apply the following identities, holding for $g,h\in G$:
\begin{align*}
(gh-1)\ =\  &(g-1) + (h-1) + (g-1)(h-1), \\
(h-1)(g-1)\  =\ & (g-1)(h-1) + \\
&(1 + (g-1) + (h-1) + (g-1)(h-1))([h,g]-1).
\end{align*}
We remark that, the element $[h,g]-1$ lying in $\Ideal(\FF G)^2$, we have
\begin{equation}\label{eq:CommCongruence}
(h-1)(g-1) \equiv (g-1)(h-1) + ([h,g] -1) \bmod\Ideal(\FF G)^3.
\end{equation}
In the upcoming proof of \cref{thm:2nd} we will sometimes omit to explicitly reference to these identities.

\begin{proof}[Proof of \cref{thm:2nd}]
Assume $\FF G \cong \FF H$.
It follows from the assumptions that $G$ belongs to the family
\[\mathcal{X}=\set{X \given X \textup{ finite $2$-group, } \Zeta(X) \textup{ cyclic, } X/\Zeta(X) \textup{ dihedral }}.\]
\cref{lem:cyclic,lem:capable,cor:invariants} guarantee that $H$ also belongs to $\mathcal{X}$.
From \cref{thm:trichotomy} we know that $G$ is isomorphic to one of $\Dih_{2^{m|n}}$, $\Quat_{2^{m|n}}$ or $\Semi_{2^{m|n}}$, for uniquely determined values of $m\geq 1$ and $n\geq 2$.
As $|G|=|H|$, we deduce from \cref{thm:WS,thm:trichotomy} that $H$ is isomorphic to one of $\Dih_{2^{m|n}}$, $\Quat_{2^{m|n}}$ or $\Semi_{2^{m|n}}$, too.

Assume first that $m = 1$. Then $G$ and $H$ are $2$-groups of maximal class and the modular isomorphism problem for these groups has been positively solved in \cite[p. 434]{Car77} and \cite[Theorem 1]{Bag92}; see also \cite{RV13} for the prime field case.

We now assume that $m > 1$. We write $\dg(G) = \log_p |G:\Frat(G)|$, which is an invariant by \cref{prop:Frat}. Since $m > 1$, we have $\Frat(\Dih_{2^{m|n}}) = \langle c^2, [b, a] \rangle$, $\Frat(\Quat_{2^{m|n}}) = \langle c, [b, a] \rangle$ and $\Frat(\Semi_{2^{m|n}}) = \langle c, [b, a] \rangle$, from which we derive $$\dg(\Dih_{2^{m|n}}) = 3 \textup{ while }\dg(\Quat_{2^{m |n}}) = \dg(\Semi_{2^{m|n}}) = 2.$$
 It follows that either $G$ and $H$ are both isomorphic to $\Dih_{2^{m|n}}$ or they are both $2$-generated. We assume that neither $G$ nor $H$ is isomorphic to $\Dih_{2^{m|n}}$. Moreover, as $\Quat_{2^{m|2}} \cong \Semi_{2^{m|2}}$, we assume
  that $n > 2$ and, for a contradiction, that
  $$G = \Quat_{2^{m|n}} \textup{ and } H = \Semi_{2^{m|n}}.$$
To produce the desired contradiction, we will show that $\Omega_1(\FF G) \subseteq \Ideal(\FF G)^2$, while this is not the case if we replace $\FF G$ by $\FF H$.

As $a$ and $b$ span $G/\Frat(G)$,
we write a generic element $x$ in $\Omega_1(\FF G)$ as
\[x \equiv \alpha (a-1) + \beta (b-1) \bmod \Ideal(\FF G)^2 \]
where $\alpha, \beta \in \FF$. It follows from \cref{lem:dream} and  \eqref{eq:CommCongruence} that
\begin{equation}\begin{split}\label{eq:square}
0 = x^2 &\equiv (\alpha^2 +\beta^2) (c-1) + \alpha\beta((a-1)(b-1) + (b-1)(a-1)), \\
 &\equiv (\alpha^2 +\beta^2) (c-1) + \alpha\beta ([b,a]-1) \bmod \Ideal(\FF G)^3.
\end{split}
\end{equation}
Set $P = G/D_3(G)$, which is the largest quotient $Q$ of $G$ satisfying $D_3(Q) = 1$.
Note that $[a, b]^{2^{n - 2}} = c^{2^{m - 1}}$ holds in $G$.
Since $m > 1$ and $n > 2$, we obtain
\begin{align*}
	P \cong \gen{ a, b, c \given a^2 = c,\ b^2 = c,\ [b, a]^2 = c^2 = [a, c] = [b, c] = 1 },
\end{align*}
which is a group of order $16$.
Then we have
\begin{align*}
	D_1(G)/D_2(G) & \cong D_1(P)/D_2(P) \cong P/\Frat(P),\\
	D_2(G)/D_3(G) & \cong D_2(P)/D_3(P) \cong \Frat(P)
\end{align*}
and both $P/\Frat(P)$ and $\Frat(P)$ have rank $2$; cf.\ \cref{fig:filt} for a choice of basis.

\begin{figure}[htbp]
	\centering
	\begin{tabular}{ccc}
		$D_1(G)/D_2(G)$ &  & $a, \ b\ \bmod D_2(G)$ \\
		$D_2(G)/D_3(G)$ & & $c,\ [b,a]\ \bmod D_3(G)$ \\
	\end{tabular}
	\caption{Quotients of the dimension subgroups and their basis.}	
	\label{fig:filt}
\end{figure}

\noindent
It follows that $c-1$ and $[b,a]-1$ are linearly independent modulo $\Ideal(\FF G)^3$.
Hence \eqref{eq:square} yields that $\alpha \beta = \alpha^2 + \beta^2 = 0$, which itself implies $\alpha = \beta = 0$. We have proven that  $x \in \Ideal(\FF G)^2$.
To conclude, note that, in $\FF H$, the element $a-1$ lies in $\Omega_1(\FF H)$ but not in $\Ideal(\FF H)^2$, because $a \notin D_2(H) = \Frat(H)$. This shows that $\FF G \not\cong \FF H$, a contradiction.
\end{proof}

\begin{cor}\label{cor:2nd}
Assume that $G$ satisfies
\begin{equation*}
	| G : \Comm(G)\Zeta(G)| = 4 \text{ and } |\Soc(G) \cap \Frat(G)| = 2.
\end{equation*}
Then the modular isomorphism problem has a positive solution for $G$.
\end{cor}
\begin{proof}
	Let $G = T \times U$ be an elementary decomposition of $G$ with $T$ elementary abelian.
	Then $| U : \Comm(U)\Zeta(U)| = 4$ and \cref{lem:exists} yields $|\Soc(U)| = 2$.
	\Cref{lem:cyclic,lem:capable} show that $U$ has cyclic center and dihedral central quotient.
	By \cref{main:reduction,thm:2nd}, the modular isomorphism problem has a positive answer.
\end{proof}

We conclude the present section and the paper with a note on the similarities between our (positive) examples and the counterexample found in \cite{GLMdR22}.

\begin{rmk}\label{rmk:counterex}
The counterexample to the modular isomorphism problem \cite{GLMdR22} shares the following properties with our applications:
\begin{itemize}
\item $|G:\Zeta(G)| = 8$ and class $3$, in \cref{thm:1st},
\item $G/\Zeta(G)$ dihedral, in \cref{thm:2nd}.
\end{itemize}
\end{rmk}

\end{document}